\documentclass[graybox]{svmult}
\usepackage{amssymb,amsmath}
\usepackage{mathptmx}       
\usepackage{helvet}         
\usepackage{courier}        
\usepackage{type1cm}        
\usepackage{makeidx}         
\usepackage{graphicx}        
\usepackage{multicol}        
\usepackage[bottom]{footmisc}
\def\be#1{\begin{equation}\label{#1}}
\def\ee{\end{equation}}
 
\def\be#1{\begin{equation}\label{#1}}
\usepackage[utf8]{inputenc}
 \newcommand{\s}[1]{\left(#1\right)}
 
 \newcommand{\n}[1]{\left\|#1\right\|}
 
 \renewcommand{\a}[1]{\left|#1\right|}

\makeindex             

\begin{document}

\title*{Heuristic parameter choice in Tikhonov method from minimizers of the quasi-optimality function
}
\author{Toomas Raus and Uno H\"amarik}
\institute{University of Tartu, \email{toomas.raus@ut.ee, uno.hamarik@ut.ee}}
\titlerunning{Heuristic parameter choice in Tikhonov regularization}
\maketitle
\abstract*{We}
\abstract{We consider choice of the regularization parameter in Tikhonov method in the case of the unknown noise level of the data. From known heuristic parameter choice rules often the best results were obtained in the quasi-optimality criterion where the parameter is chosen as the global minimizer of the quasi-optimality function. In some problems this rule fails, the error of the Tikhonov approximation is very large. We prove, that one of the local minimizers of the quasi-optimality function is always a good regularization parameter. We propose an algorithm for finding a proper local minimizer of the quasi-optimality function.}

\section{Introduction}

Let  $A\in{\cal L}(H,F)$ be a linear bounded operator between real Hilbert spaces. We are interested in finding the minimum norm solution $u_*$ of the equation
\begin{equation}\label{mitteek_1}
Au=f_*,\qquad f_*\in R(A).
\end{equation}
The range ${\cal R}(A)$ may be non-closed and the kernel ${\cal N}(A)$ may be non-trivial, so in general this problem is ill-posed. As usually in treatment of ill-posed problems, we assume that instead of exact data $f_*$ noisy data $f\in F$ are given. 
For the solution of the problem $Au=f$ we consider 
Tikhonov method (see \cite{EHN, VaiVer}) where regularized solutions in cases of exact and inexact data have corresponding forms
\[
u^+_{\alpha}=\s{\alpha I+A^*A}^{-1}A^*f_*, \qquad u_{\alpha}=\s{\alpha I+A^*A}^{-1}A^*f
\]
and $\alpha>0$ is the regularization parameter. \\
Denote
\be{e1}
 e_1(\alpha):=\n{u^+_{\alpha}-u_*}+\n{u_{\alpha}-u^+_{\alpha}}.
\ee
Due to the well-known estimate $\n{u_{\alpha}-u^+_{\alpha}} \leq \frac12 {\alpha}^{-1/2}\n{f-f_*}$ (see \cite{EHN, VaiVer}) the error
$\n{u_{\alpha}-u_*}$ can be estimated by
\be{err-est}
\n{u_{\alpha}-u_*} \leq e_1(\alpha) \leq e_2(\alpha,\n{f-f_*}):=\n{u^+_{\alpha}-u_*}+\frac{1}{2 \sqrt \alpha} \n{f-f_*}.
\ee 
We consider choice of the regularization parameter if the noise level for $\n{f-f_*}$ is unknown. The parameter choice rules which do not use the noise level information are called heuristic rules.  Many heuristic rules are proposed, well known are the quasi-optimality criterion \cite{BaKi8, BaKi9, BaRe, HPR09, Ki, Ki13, KiNe, Ne, TGK}, L-curve rule \cite{Ha92, Ha94}, GCV-rule \cite{GHW}, Hanke-Raus rule \cite{HR96}, Reginska's rule \cite{Re}, about other rules see \cite{HoReRo, JiLo, LuMa, Pa}. 
Heuristic rules are numerically compared in \cite{BaLu, HPR09, HoReRo, Pa}. It is also well known that it is not possible to construct heuristic rule guaranteeing convergence $\n{u_{\alpha}-u_*} \to 0$ as the noise level goes to zero (see \cite{Bak}). Nevertheless the heuristic rules give good results in many problems. The problem is that all these rules may fail in some problems and without additional information about the solution, it is difficult to decide, is the obtained parameter reliable or not.

In this article we propose a new strategy for heuristic parameter choice. It is based on analysis of local minimizers of the function $\psi_Q(\alpha)=\alpha\n{\frac{du_\alpha}{d\alpha}}$, the global minimizer of which on certain interval $[\alpha_M, \alpha_0]$ is taken for parameter in the quasi-optimality criterion. We will call the parameter $\alpha_R$ in arbitrary rule R as pseudooptimal, if $$\n{u_{\alpha_R}-u_*} \leq \text{const} \quad \min_{\alpha>0} e_1(\alpha)$$ and we show that at least one of local minimizers  of $\psi_Q(\alpha)$ has this property. Our approach enables to replace the search of the parameter from the interval $[\alpha_M, \alpha_0]$ by search of the proper parameter from the set $L_{min}$ of the local minimizers of the function $\psi_Q(\alpha)$.
We consider also the possibility to restrict the set $L_{min}$ to its subset $L^*_{min}$ still containing at least one pseudooptimal parameter. It occurs that in many problems the restricted set $L^*_{min}$ contains only one local minimizer and this is the pseudooptimal parameter. If the set $L^*_{min}$ contains several local minimizers, we consider different algorithms for choice of the proper parameter from the set 
$L^*_{min}$. 

The plan of this paper is as follows. In Section 2 we consider known rules for choice of the regularization parameter, both in case of known and unknown noise level. We will characterize distinctive properties of considered heuristic rules presenting results of numerical experiments on test problems \cite{Ha94}. In Section 3 we consider the set $L_{min}$ of local minimizers of the function $\psi_Q(\alpha)$ and prove that this set contains at least one pseudooptimal parameter. In Section 4 we show how to restrict the set $L_{min}$ to the set $L^*_{min}$ still containing at least one pseudooptimal parameter.
In Section 5 we consider the case if the set $L^*_{min}$ contains  several elements and we propose some algorithms for finding proper pseudooptimal parameter. In all sections theoretical results and proposed algorithms are illustrated by results of numerical experiments on test problems \cite{Ha94}.
\\

\section{Rules for the choice of the regularization parameter}

An important problem, when applying regularization methods, is the proper choice of the regularization parameter. The choice of the parameter  depends on the information about the noise level. 

\subsection{Parameter choice in the case of known noise level}
In case of known noise level $\delta, \n{f-f_*} \leq \delta$ we use one of so-called  $\delta$-rules, where certain functional $d(\alpha)$ and constants $b_2 \geq b_1 \geq b_0$ ($b_0$ depends on $d(\alpha)$) are chosen and such regularization parameter $\alpha(\delta)$ is chosen which satisfies 
$b_1 \delta \leq d(\alpha) \leq b_2 \delta. $ \\
1) Discrepancy principle (DP) \cite{Mo, VaiVer}: $$b_1 \delta \leq \n{Au_{\alpha}-f} \leq b_2 \delta,  \quad b_1 \geq 1.$$ 
2) Modified discrepancy principle (Raus-Gfrerer rule) \cite{Gf, Ra85}: 
$$b_1 \delta \leq \n{B_{\alpha}\s{Au_{\alpha}-f}} \leq b_2 \delta, \quad B_{\alpha}:=\alpha^{1/2} \s{\alpha I+AA^*}^{-1/2}, \quad b_1 \geq 1
 .$$ 
3) Monotone error rule (ME-rule) \cite{HKPRT, TaHa}: 
$$b_1 \delta \leq  \frac{\n{B_{\alpha}\s{Au_{\alpha}-f}}^2}{\n{B^2_{\alpha}\s{Au_{\alpha}-f}}}\leq b_2 \delta,  \quad b_1 \geq 1.$$
The name of this rule is justified by the fact that
the chosen parameter $\alpha_{\text{ME}}$ satisfies
$$\frac{d}{d\alpha} \| u_{\alpha}-u_*\| >0 \qquad\forall \alpha \in (\alpha_{\text{ME}}, \infty).$$
Therefore $\alpha_{\text{ME}} \geq \alpha_{opt}:=\text{argmin} \|u_\alpha -u_*\|$ and $b_1=b_2=1$ are recommended. 

4) Monotone error rule with post-estimation (MEe-rule) \cite{HPR09, HPR11, HPR12, Pa, RH09a}. The inequality $\alpha_{ME} \geq \alpha_{opt}$ suggests to use somewhat smaller parameter than $\alpha_{ME}$. Extensive numerical experiments suggest to take $b_1=b_2=1$, to compute $\alpha_{\text{ME}}$ and to use the post-estimated parameter $\alpha_{MEe} :=0.4 \alpha_{ME}$. Then typically $\|u_{\alpha_\text{MEe}} - u_*\| \approx 0.8 \|u_{\alpha_\text{ME}} - u_*\|$. To our best knowledge in case of exact noise level this MEe-rule gives typically best results from all known rules for the parameter choice.

5) Rule R1 \cite {Ra92}: Let $b_2 \geq b_1 \geq 0.325$. Let $d(\alpha):=
\alpha^{-1/2}\n{A^*B^2_{\alpha}\s{Au_{\alpha}-f}}$. Choose $\alpha(\delta)$ such that $d(\alpha(\delta)) 
\geq b_1 \delta$, but $d(\alpha) \leq b_2 \delta$ for all $\alpha \leq \alpha(\delta)$.

Note that 
$$B^2_{\alpha}\s{Au_{\alpha}-f}=Au_{2,\alpha}-f, \qquad u_{2,\alpha}=\s{\alpha I+A^*A}^{-1}(\alpha u_\alpha+A^*f),$$
where $u_{2,\alpha}$ is the 2-iterated Tikhonov approximation.

6) Balancing principle \cite{BaLu, HR09, Pa, PeSc}. This rule has different forms in different papers, in \cite{HR09} the form 
$$b_1 \delta \leq \frac{\sqrt{\alpha}\sqrt{q} \|u_\alpha - u_{\alpha/q}\|}{1-q} \leq b_2 \delta, \quad b_1\geq \frac{3\sqrt 6}{16} \approx 0.459.$$
Typically balancing principle is implemented by computing a sequence of Tikhonov approximations, but in case of a smooth solution much better approximation than single Tikhonov approximation is simple linear combination of Tikhonov approximations with different parameters - the extrapolated approximation (see \cite{HPR10, HR09, Pa}).
See \cite{RH09b} about effective numerical realization of rules 1)-6).

The last five rules are weakly quasioptimal rules  (see \cite{RH07}) for Tikhonov method.
If $\n{f-f_*} \leq \delta$, then we have the error estimate (see (\ref{err-est}))
\[
\n{u_{\alpha(\delta)}-u_*} \leq C(b_1, b_2) \inf_{\alpha>0} e_2(\alpha,\delta)=C(b_1, b_2) \inf_{\alpha>0} \left[\n{u^+_{\alpha}-u_*}+\frac{1}{2 \sqrt \alpha}\delta\right].
\]

\subsection {Parameter choice in the case of unknown noise level}
 If the noise level is unknown, then, as shown by Bakushinskii [1], no rule for choosing the regularization parameter can guarantee the convergence of the regularized solution to the exact one as noise level $\n{f-f_*}$ goes to zero. 
Nevertheless, some heuristic rules are rather popular, because they often work well in practice and because in applied ill-posed problems the exact noise level is often unknown. \\
 A classical heuristic rule is the quasi-optimality criterion. In Tikhonov method it chooses  $\alpha=\alpha_Q$ as the global minimizer of the function
\be{qf}
\psi_Q(\alpha)=\alpha\n{\frac{du_\alpha}{d\alpha}}={\alpha}^{-1}\n{A^*B^2_{\alpha}\s{Au_{\alpha}-f}}.
\ee
In case of the discrete version of the  quasi-optimality criterion we choose  $\alpha=\alpha_{QD}$ as the global minimizer of the function $\n{u_{\alpha}-u_{q\alpha}}$, where $0<q<1$. \\
The Hanke-Raus rule finds the regularization parameter  $\alpha=\alpha_{HR}$ as the global minimizer of the function
\[
\psi_{HR}(\alpha)={\alpha}^{-1/2}\n{B_{\alpha}\s{Au_{\alpha}-f}}.
\]
In practice the L-curve rule is popular. This rule uses the graph with log-log scale, on $x$-axis $\n{Au_{\alpha}-f}$ and on $y$-axis $\n{u_{\alpha}}$. The name of the rule is justified by fact that often the points $\s{\n{Au_{\alpha}-f},\n{u_{\alpha}}}$ have shape similar to the letter L and parameter $\alpha_L$ which corresponds to the "corner point" is often a good parameter. In the literature several concrete rules for choice of the 'corner point'  are proposed. One natural rule is proposed in \cite{Re}
where global minimum point of the function
 \[
\psi_{RE}(\alpha)=\n{Au_{\alpha}-f}\n{u_{\alpha}}^{\tau}, \qquad \tau \geq 1.
\]
In numerical examples below we used this rule with $\tau=1$.\\

Some heuristic rules choose the regularization parameter as global minimizer of a function $\alpha^{-1/2} d(\delta)$ with function $d(\delta)$ from some $\delta$-rule 1)-6) from Section 2.1 (see \cite{HPR09}). For example, the quasi-optimality criterion and Hanke-Raus rule use functions $d(\delta)$ from the rules 5) (R1) and 2) (modified discrepancy principle) respectively. In \cite{HPR09} heuristic counterpart of rule 3) (ME-rule) is also studied. We call this rule as HME-rule (H means "heuristic counterpart"), here 
the regularization parameter  $\alpha=\alpha_{HME}$ is chosen as the global minimizer of the function
\[
\psi_{HME}(\alpha)={\alpha}^{-1/2}\frac{\n{B_{\alpha}\s{Au_{\alpha}-f}}^2}{\n{B^2_{\alpha}\s{Au_{\alpha}-f}}}.
\]  
In the following we will find the regularization parameter from the set of parameters
\be{omega}
\Omega=\left\{\alpha_j: \alpha_j=q\alpha_{j-1}, \quad j=1,2,...,M, \quad 0<q<1  \right\}
\ee
where $\alpha_0, q, \alpha_M$ are given.  
In the case if in the discretized problem the minimal eigenvalue $\lambda_{min}$ of the matrix
 $A^TA$ is larger than $\alpha_M$, the heuristic rules above choose parameter
 $\alpha_M$, which is generally not a good parameter. The works \cite{Ki13, KiNe, Ne} propose to search the global minimum of the function  $\psi_Q(\alpha)$ in the interval $[\max{\s{\alpha_M,\lambda_{min}}},\alpha_0]$. 
We use basically the same approach but consider also local minimizers. 

We say that the discretized problem
$Au=f$ do not need regularization if
$$e_1(\lambda_{min})=\min_{\alpha \in \Omega, \alpha \geq \lambda_{min}} e_1(\alpha).$$ 
If $\lambda_{min} > \alpha_M$ and the discretized problem do not need regularization then $\alpha_M$ is the proper parameter while then it is easy to show the error estimate
$$ \n{u_{\alpha_M}-u_*} \leq e_1(\alpha_M) \leq 2 \min_{\alpha \in \Omega} e_1(\alpha).$$  Searching the parameter from the interval $[\max{\s{\alpha_M,\lambda_{min}}},\alpha_0]$ means the a priori assumption that the discretized problem needs regularization. Note that if $\lambda_{min} > \alpha_M$, then in general case it is not possible to decide (without additional information about solution or about noise of the data), needs the discretized problem regularization or not. In practice in the case
 $\lambda_{min} > \alpha_M$ it is meaningful to choose the regularization parameter $\alpha_H$ from the interval $[\lambda_{min}, \alpha_0]$, while then our parameter is not too small. If we have some information about solution or about the noise then this information may help to decide, is $\alpha_H$ or $\alpha_M$ the better final parameter.

Our tests are performed on the well-known set of test problems by Hansen \cite{Ha94}.
In all tests we used discretization parameter $n=100$. \\
Since the perfomance of rules generally depends on the smoothness $p$ of the exact solution in (\ref{mitteek_1}), we complemented the standard solutions $u_*$ of (now discrete) test problems with smoothened solutions $|A|^p u_*, |A|:=(A^*A)^{1/2}, p=2$ (computing the right-hand side as $A(|A|^p u_*)$). After discretization all problems were scaled (normalized) in such a way that the Euclidean norms of the operator and the right-hand side were 1. On the base of exact data $f_*$ we formed the noisy data $f$, where $\n{f-f_*}$ has values $10^{-1}, 10^{-2},...,10^{-6}$, $f-f_*$ has normal distribution and the components of the noise were uncorrelated. We generated 20 noise vectors and used these vectors in all problems. We search the regularization parameter from the set 
$\Omega$, where $\alpha_0=1, q=0.95 $ and $M$ is chosen so that  $\alpha_M \geq 10^{-18} > \alpha_{M+1}$. \\
Since in model equations the exact solution is known, it is possible to find the regularization parameter $\alpha_*$, which gives the smallest error in the set $\Omega$ . 
For every rule R the error ratio $$E=\frac{\n{u_{\alpha_R}-u_*}}{\n{u_{\alpha_*}-u_*}}= \frac{\n{u_{\alpha_R}-u_*}}{\min_{\alpha \in \Omega} \n{u_{\alpha}-u_*}}$$
describes the performance of the rule R on this particular problem. To compare the rules or to present their properties, the following tables show averages A and maximums M of these error ratios over various parameters of the data set (problems 1-10, smoothness indices $p$, noise levels $\delta$).  We say that the heuristic rule fails if the error ratio $E>100$. 
Table 1 contains the results of the previous heuristic rules by problems. \\

\begin{table}
\caption{Averages of error ratios E and fail \% (in parenthesis) for heuristic rules, $p=0$}
\label{tab:1}       
%
%
\addtolength\tabcolsep{9.8pt}%
\begin{tabular}{lccccc}
\hline\noalign{\smallskip}
Problem & $\Lambda$ & Quasiopt. & HR  & HME & Reginska \\
\noalign{\smallskip}\svhline\noalign{\smallskip}

Baart & 1666  & 1.54 & 2.58 & 2.52 & 1.32\\
Deriv2 & 16 &  1.08 & 2.07 & 1.72 & 35.19 (3.3)\\
Foxgood & 210  &  1.57 & 8.36 & 7.71 & 36.94 (10.8) \\
Gravity & 4  &  1.13 & 2.66 & 2.32 & 20.49 (0.8)\\
Heat & $4*10^{29}$  &  $>100$ (66.7)& 1.64 & 1.48 & 23.40 (4.2)\\
Ilaplace & 16  &  1.24 & 1.94 & 1.81 & 1.66\\
Phillips & 9  &  1.09 & 2.27 & 1.91 & $>100$ (44.2)\\
Shaw & 290  &  1.43 & 2.34 & 2.23 & 1.80\\
Spikes & 1529  &  1.01 & 1.03 & 1.03 & 1.01\\
Wing & 9219  &  1.40 & 1.51 & 1.51 & 1.18\\
\noalign{\smallskip}\hline\noalign{\smallskip}
\end{tabular}
\end{table}

This table shows that the quasi-optimality principle succeeds to choose a proper parameter in almost all problems, except the problem \emph{heat} where this principle fails in 66.7 \% cases. In contrast to other problems in problem \emph{heat} the maximal ratio
$\Lambda=\max_{ \lambda_k>\max{(\alpha_M,\lambda_n)}} \lambda_k/\lambda_{k+1} $
of consecutive eigenvalues $\lambda_1 \geq \lambda_2 \geq ... \geq \lambda_n$ of the matrix $A^TA$ in the interval 
$[\max{(\alpha_M,\lambda_n)},1]$ is much larger than in other problems. It means that location of the eigenvalues  in the interval $[\max{(\alpha_M,\lambda_n)},1]$ is sparse. 

The rules of Hanke-Raus and HME did not fail in test problems, but the error of the approximate solution is in most problems approximately two times larger than for parameter chosen by the quasi-optimality principle.
The problem in these rules is that they choose too large parameter comparing with the optimal parameter.
Reginska's rule may fail in many problems but it has the advantage that it works better than other rules if the noise level is large. The Reginska's rule has average of error ratios of all problems $E=1.46$ and   $E=3.23$ in cases $\n{f-f_*}=10^{-1}$ and $\n{f-f_*}=10^{-2}$ respectively, the Hanke-Raus rule has corresponding averages $E=3.41$  ja $E=3.50$.

By implementing of all these rules the problem is that without additional information in general case it is difficult  to decide, is the obtained parameter good or not. In the following we propose a methodology enabling in many cases to assert that obtained parameter is pseudooptimal.

\section {Local minimum points of the function  $\psi_{Q}(\alpha)$}

In the following we investigate the function
 $\psi_{Q}(\alpha)$ in (\ref{qf}) and show that at least one local minimizer of this function is the pseudooptimal parameter. 
We need some preliminary results. 

\begin{lemma}
The function  $\psi_{Q}(\alpha)$ has the estimate (see (\ref{e1}) for notation $e_1(\alpha)$)
\begin{equation} \label{eq3}
\psi_{Q}(\alpha) \leq e_1(\alpha).
\end{equation}
\end{lemma}

\begin{proof}
The following equalities hold:
$$Au_{\alpha}-f=A\s{\alpha I+A^*A}^{-1}A^*f -f =-\alpha \s{\alpha I+AA^*}^{-1}f,
$$
\begin{eqnarray}\label{qfe}
(\alpha)^{-1}A^*B^2_{\alpha}\s{Au_{\alpha}-f}=-\alpha A^*\s{\alpha I+AA^*}^{-2}f=-\alpha \s{\alpha I+A^*A}^{-2}A^*f=\\ \nonumber
=\alpha A^*A \s{\alpha I+A^*A}^{-2}u_* + \alpha \s{\alpha I+A^*A} ^{-2}A^*(f-f_*).
\end{eqnarray}
Now the inequality (\ref{eq3}) follows from (\ref{qf}) and the inequalities 
\[
\alpha \n{A^*A\s{\alpha I+A^*A}^{-2}u_*} \leq \alpha \n{\s{\alpha I+A^*A}^{-1}u_*}=\n{u^+_{\alpha}-u_*}.
\]
\[
\alpha \n{\s{\alpha I+A^*A} ^{-2}A^*(f-f_*)} \leq   \n{\s{\alpha I+A^*A} ^{-1}A^*(f-f_*)}=\n{u_{\alpha}-u^+_{\alpha}}.
\]
\end{proof}
\begin{remark} Note that $\lim_{\alpha \rightarrow \infty} \psi_{Q}(\alpha) =0 $, but  $\lim_{\alpha \rightarrow \infty} e_1(\alpha) =\n{u_*}$. 
Therefore in the case of too large $\alpha_0$ this $\alpha_0$ may be global (or local) minimizer of the function $\psi_{Q}(\alpha)$. We recommend to take 
$\alpha_0=c\n{A^*A}, c \leq 1$ or to minimize the function  $\tilde{\psi}_{Q}(\alpha):=(1+\alpha / \n{A^*A}) \psi_{Q}(\alpha)$
instead of $\psi_{Q}(\alpha)$. Due to limit
$\lim_{\alpha \rightarrow 0} (1+\alpha / \n{A^*A})=1$ the function $\tilde{\psi}_{Q}(\alpha)$ approximately satisfies  (\ref{eq3}).
\end{remark}
\begin{lemma}
Denote $\psi_{QD}(\alpha)=\s{1-q}^{-1}\n{u_{\alpha}-u_{q\alpha}}$. Then it holds
\[
\psi_{Q}(\alpha) \leq \psi_{QD}(\alpha) \leq q^{-1}\psi_{Q}(q\alpha).
\]
\end{lemma}
\begin{proof} We use the equalities (\ref{qfe}) and
\begin{eqnarray*}
u_{\alpha}-u_{q\alpha}=\s{\alpha I+A^*A}^{-1}A^*f-\s{q\alpha I+A^*A}^{-1}A^*f=\\
\s{q-1}\alpha \s{\alpha I+A^*A}^{-1}\s{q\alpha I+A^*A}^{-1} A^*f.
\end{eqnarray*}
The following inequalities prove the lemma: 
$$\psi_{Q}(\alpha)=\alpha  \n{\s{\alpha I+A^*A}^{-2}A^*f} \leq \alpha  \n{\s{\alpha I+A^*A}^{-1} \s{q\alpha I+A^*A}^{-1} A^*f} =$$
$$=\psi_{QD}(\alpha) \leq \alpha  \n{\s{q\alpha I+A^*A}^{-2}A^*f}=q^{-1}\psi_{Q}(q\alpha).$$
\end{proof}
In the following we define the local minimum points of the function  $\psi_Q(\alpha) $ on the set $\Omega$ (see (\ref{omega})).

We say that the parameter $\alpha_k , 0 \leq k \leq M-1 $ is the local minimum point of the sequence $\psi_Q(\alpha_k) $, if  $\psi_Q(\alpha_k) <\psi_Q(\alpha_{k+1})$ and in case $k>0 $ there exists index $j \geq 1$ such, that  $\psi_Q(\alpha_k) =\psi_Q(\alpha_{k-1}) =...=\psi_Q(\alpha_{k-j+1}) <\psi_Q(\alpha_{k-j})$.
 The parameter $\alpha_M$ is the local minimum point if there exists index $j \geq 1$ so, that  $$\psi_Q(\alpha_M) =\psi_Q(\alpha_{M-1}) =...=\psi_Q(\alpha_{M-j+1}) <\psi_Q(\alpha_{M-j}).$$ Let the number of the local minimum points be $K$ and denote
\[
L_{min}=\left\{\alpha^{(k)}_{min}: \alpha^{(1)}_{min}>\alpha^{(2)}_{min}>...>\alpha^{(K)}_{min}\right\}.
\]
The parameter $\alpha_k , 0 < k < M $ is the local maximum point of the sequence $\psi_Q(\alpha_k) $ if  $\psi_Q(\alpha_k) >\psi_Q(\alpha_{k+1})$ and there exists index $j \geq 1$ so, that  $$\psi_Q(\alpha_k) =\psi_Q(\alpha_{k-1})  =...=\psi_Q(\alpha_{k-j+1}) >\psi_Q(\alpha_{k-j}).$$
We denote by $\alpha^{(k)}_{max}$ the local maximum point between the local minimum points  $\alpha^{(k+1)}_{min}$ and  $\alpha^{(k)}_{min},  1 \leq k \leq K-1$. Denote $\alpha^{(0)}_{max}=\alpha_0, \alpha^{(K)}_{max}=\alpha_M $. Then by the construction
\[
 \alpha^{(0)}_{max} \geq \alpha^{(1)}_{min}> \alpha^{(1)}_{max}>...>\alpha^{(K-1)}_{max}>\alpha^{(K)}_{min} \geq \alpha^{(K)}_{max}.
\]
\begin{theorem}
The following estimates hold for the local minimum points of the function $\psi_Q(\alpha) $:
\begin{enumerate}
\item
\begin{equation} \label{eq6} 
 \min_{\alpha \in L_{min}} \n{u_{\alpha}-u_*} \leq q^{-1}C \min_{\alpha_M \leq \alpha \leq \alpha_0} e_1(\alpha),   
\end{equation} 
where 
\begin{equation*} 
C:=1+\max_{1 \leq k \leq K} \max_{\alpha_j  \in \Omega, \alpha^{(k)}_{max} \leq \alpha_j \leq \alpha^{(k-1)}_{max}} T\s{\alpha^{(k)}_{min}, \alpha_j} \leq 1+c_q\ln \s{\frac{\alpha_0}{\alpha_M}}, 
\end{equation*} 
$$T(\alpha, \beta):=\frac{\n{u_{\alpha}-u_{\beta}}}{\psi_Q(\beta)}, \qquad c_q:=\s{q^{-1}-1}/\ln{q^{-1}} \to 1\text{  if  } q \to 1.$$

\item Let $u_*=\a{A}^pv$, $\n{v}\leq\rho$, $p>0$ and $\alpha_0=1$. 
If $\delta_0:=\sqrt \alpha_M \leq \n{f-f_*}$, then 
\begin{equation} \label{eq7}
\min_{\alpha \in L_{min}} \n{u_{\alpha}-u_*} \leq   c_p\ln \frac{\n{f-f_*}}{\delta_0} \rho^\frac1{p+1}\a{\ln{\n{f-f_*}}}\n{f-f_*}^\frac p{p+1}, 0<p\leq 2.
\end{equation} 
\end{enumerate}
\end{theorem}
\begin{proof}
For arbitrary parameters $\alpha \geq 0, \quad \beta \geq 0$ the inequalities
$$\n{u_{\alpha}-u_*} \leq \n{u_{\alpha}-u_{\beta}} + \n{u_{\beta}-u_*} \leq  
T(\alpha, \beta)\psi_Q(\beta)+e_1(\beta)$$
and (\ref{eq3}) lead to the estimate
\begin{equation} \label {eq4}
\n{u_{\alpha}-u_*} \leq \s{1+T(\alpha, \beta)}e_1(\beta).
\end{equation}
It is easy to see that 
\begin {equation}  \label {eq5}
\min_{\alpha_j \in \Omega} e_1(\alpha_{j}) \leq q^{-1}\min_{\alpha_M \leq \alpha \leq \alpha_0} e_1(\alpha),
\end {equation}
while in case $ q\alpha \leq \alpha' \leq \alpha$ we have $e_1\s{\alpha'} \leq  q^{-1} e_1\s{\alpha}$. \\
Let ${\alpha}_{j*}=\alpha_0 q^{j*}$ be the global minimum point of the function $e_1(\alpha)$ on the set of the parameters $\Omega$. 
Then ${\alpha}_{j*} \in [\alpha^{(k)}_{max},\alpha^{(k-1)}_{max}] $ for some $k, 1  \leq k \leq K$. Denote  $u_j=u_{{\alpha}_{j}}$ and $u_{kmin}=u_{{\alpha}^{(k)}_{min}}$. Then using  (\ref{eq4})  we can estimate
\begin {eqnarray*}
\n{u_{kmin}-u_*} \leq \s{1+T(\alpha^{(k)}_{min}, \alpha_{j*})}e_1(\alpha_{j*}) \leq \\ \s{1+\min_{\alpha^{(k)}_{max} \leq \alpha_j \leq \alpha^{(k-1)}_{max}}T(\alpha^{(k)}_{min}, \alpha_{j})} \min_{\alpha_j \in \Omega} e_1(\alpha_{j}).
\end {eqnarray*}
Since we do not know to which interval $[\alpha^{(k)}_{max},\alpha^{(k-1)}_{max}] $ the parameter ${\alpha}_{j*}$ belongs, we take maximum of $T$ over all intervals, $ 1 \leq k \leq K$. Using also (\ref{eq5}) we obtain the estimate (\ref{eq6}). \\
Now we show that  $C \leq 1+c_q\ln \s{\frac{\alpha_0}{\alpha_M}}$. At first we estimate $T(\alpha^{(k)}_{min}, \alpha_{j})$ in the case if  $\alpha^{(k)}_{min} \leq \alpha_{j} \leq \alpha^{(k-1)}_{max}$. Then Lemma 2 enables to estimate 
\begin {eqnarray*}
\n{u_{kmin}-u_j} \leq \Sigma_{j \leq i \leq kmin-1}\n{u_{i}-u_{i+1}} \leq q^{-1} (1-q) \Sigma_{j \leq i \leq kmin-1} \psi_Q(\alpha_{i+1})
\end {eqnarray*}
and
\begin {eqnarray*}
T(\alpha^{(k)}_{min}, \alpha_{j}) =\frac{\n{u_{kmin}-u_j}}{\psi_Q(\alpha_{j})} \leq 
 q^{-1} (1-q)\Sigma_{j \leq i \leq kmin-1} \frac{\psi_Q(\alpha_{i+1})}{\psi_Q(\alpha_{j})} \leq \\ (q^{-1}-1)(kmin-j) \leq  (q^{-1}-1) M= \frac{(q^{-1}-1)}{\ln{q^{-1}}}\ln{\frac{\alpha_0}{\alpha_M}}=c_q\ln{\frac{\alpha_0}{\alpha_M}}.
\end {eqnarray*}
If $\alpha^{(k)}_{max} \leq \alpha_{j} \leq \alpha^{(k)}_{min}$, then analogous estimation of  $T(\alpha^{(k)}_{min}, \alpha_{j})$  gives the same result. \\
For source-like solution $u_0-u_*=\a{A}^pv$, $\n{v}\leq\rho$, $p>0$ the error estimate
\begin {eqnarray*}
\min_{\alpha_M \leq \alpha \leq \alpha_0} e_1(\alpha) \leq c_p {\rho}^{1/(p+1)}{\n{f-f_*}}^{p/(p+1)},  0 < p \leq 2
\end {eqnarray*}
is well-known (see \cite{EHN, VaiVer})
and the equality $\ln{\frac{\alpha_0}{\alpha_M}} = \ln{{\delta_0}^{-2}} = 2\ln{\frac{\n{f-f_*}}{\delta_0}} \a{\ln{\n{f-f_*}}}$ leads to the estimate (\ref{eq7}).
\end{proof}
The results of numerical experiments for local minimizers $\alpha \in L_{min}$ of the function
 $\psi_{Q}(\alpha)$ are given in the Table 2. For comparison the results of $\delta$-rules with $\delta=\n{f-f_*}$ are added to the columns 2-4. 
Columns 5 and 6 contain respectively the averages and maximums of error ratios $E$ for the best local minimizer $\alpha \in L_{min}$. The results show that the Tikhonov approximation with the best local minimizer $\alpha \in L_{min}$ is even more accurate than with the best $\delta$-rule parameter $\alpha_\text{MEe}$. Columns 7 and 8 contain the averages and maximums of cardinalities $|L_{min}|$ of sets 
$L_{min}$ (number of elements of these sets). Note that number of local minimizers depends on parameter
$q$ (for smaller $q$ the number of local minimizers is smaller) and on length of minimization interval determined by the parameter $\alpha_M$. The number of local minimizers is smaller also for larger noise size. 
Columns 9 and 10 contain the averages and maximums of values of constant
$C$ in the a posteriori error estimate (\ref{eq6}). The value of $C$ and error estimate (\ref{eq6}) allow to assert, that in test problems \cite{Ha94} the choice of $\alpha$ as the best local minimizer in $L_{min}$ guarantees that error of the Tikhonov approximation has the same order as $\min_{\alpha_M \leq \alpha \leq \alpha_0} e_1(\alpha)$. Note that average and maximum of error ratio
$E1=\n{u_{\alpha_R}-u_*}/\min_{\alpha \in \Omega} e_1(\alpha)$ for the best local minimizer
$\alpha_R$ over all problems were 0.84 and 1.39 (for the MEe-rule corresponding error ratios were 0.85 and 1.69). 
\begin{table}
\caption{Results about the set $L_{min}$, $p=0$}
\label{tab:1}       
\addtolength\tabcolsep{4.6pt}%
\begin{tabular}{lccc|cc|cc|cc}
\hline\noalign{\smallskip}
Problem & ME & MEe & DP  & \multicolumn{2}{c|}{Best of $L_{min}$}  &\multicolumn{2}{c|}{ $|L_{min}|$} &\multicolumn{2}{c}{ Apost.  $C$ }\\
 & Aver E & Aver E & Aver E  & Aver E & Max E  & Aver & Max  & Aver & Max\\
\noalign{\smallskip}\svhline\noalign{\smallskip}

Baart & 1.43  & 1.32 & 1.37 & 1.23 & 2.51 & 6.91 & 8 & 3.19 & 3.72 \\
Deriv2 & 1.09 &  1.08 & 1.28 & 1.08 & 1.34 & 2.00 & 2 & 3.54 & 4.49 \\
Foxgood & 1.98 &  1.42 & 1.34 & 1.47 & 6.19 & 3.63 & 6 & 3.72 & 4.16 \\
Gravity & 1.40  &  1.13 & 1.16 & 1.13 & 1.83 & 1.64 & 3 & 3.71 & 4.15 \\
Heat & 1.19 & 1.03 & 1.05 & 1.12 & 2.36 & 3.19 & 5 & 3.92 & 4.50 \\
Ilaplace & 1.33 & 1.21 & 1.26 & 1.20 & 2.56 & 2.64 & 5 & 4.84 & 6.60 \\
Phillips & 1.27 & 1.02 & 1.02 & 1.06 & 1.72 & 2.14 & 3 & 3.99 & 4.66 \\
Shaw & 1.37 & 1.24 & 1.28 & 1.19 & 2.15 & 4.68 & 7 & 3.48 & 4.43 \\
Spikes & 1.01  & 1.00 & 1.01 & 1.00 & 1.02 & 8.83 & 10 & 3.27 & 3.70 \\
Wing & 1.16 & 1.13 & 1.15 & 1.09 & 1.38 & 5.20 & 6 & 3.07 & 3.72 \\
\noalign{\smallskip}\hline\noalign{\smallskip}
Total  & 1.32 & 1.16 & 1.19 & 1.16 & 6.19 & 4.09 & 10 & 3.67 & 6.60 \\
\noalign{\smallskip}\hline\noalign{\smallskip}
\end{tabular}
\end{table}

\section {Restricted set of the local minimizers of the function  $\psi_{Q}(\alpha)$}

We will restrict the set $L_{min}$ using two phases. In the first phase we remove from $L_{min}$ local minimizers in interval, where the function $\n{B_{\alpha}\s{Au_{\alpha}-f}}$ decreases only a little bit.
On the second phase we remove from set obtained on the first phase these local minimizers for which the function
  $\psi_{Q}(\alpha)$ for decreasing $\alpha$-values has only small growth before the next decrease.\\

 1. Denote $\delta_M:=\n{B_{\alpha_M}\s{Au_{\alpha_M}-f}}$ and by $\alpha=\alpha_{MD}$ the parameter for which $\n{B_{\alpha}\s{Au_{\alpha}-f}}=b\delta_M,  \quad b > 1$. Denote $\alpha_{MDQ}:=\min{\s{\alpha_{MD}, \alpha_{Q}}}$, where $\alpha_{Q} \in L_{ min}$ is the global minimizer of the function  $\psi_{Q}(\alpha)$ on the set $\Omega$. Let $ \alpha^{(k_0)}_{max} \leq \alpha_{MDQ} < \alpha^{(k_0-1)}_{max}$ for some $k_0, 1 \leq k_0 \leq K$. 
Then the set of local minimizers what we obtain on the first phase of restriction, has the form 
$L^0_{min}=\left\{\alpha^{(k)}_{min}: 1 \leq k  \leq k_0 \right\}$. In the case  $ \alpha^{(k0)}_{max} \leq \alpha_{MDQ} \leq \alpha^{(k0)}_{min}$ we change denotation to  $ \alpha^{(k0)}_{max}:=\alpha^{(k0)}_{min}$. \\

2. We remove from the set $L^0_{min}$ these local minimizers $\alpha^{(k)}_{min}$ and following maximizers $\alpha^{(k)}_{max}$, which satisfy the following conditions:

$$\alpha^{(k)}_{min} \not=\alpha^{(k)}_{max}; \qquad
\frac{\psi_{Q}(\alpha^{(k)}_{max})}{\psi_{Q}(\alpha^{(k)}_{min})} \leq c_0; \qquad
\frac{\psi_{Q}(\alpha^{(k)}_{min})}{\min_{j \leq k} {\psi_{Q}(\alpha^{(j)}_{min})}} \leq c_0 ,$$
where $c_0>1$ is some constant.  We denote by   $$L^*_{min}:=\left\{\overline \alpha^{(k)}_{min}:  \overline\alpha^{(1)}_{min} > \overline\alpha^{(2)}_{min} >... >\overline\alpha^{(k_*)}_{min}\right\}$$  the set of minimizers remained in $L^0_{min}$ and denote the remained maximizers by $\overline \alpha^{(k)}_{max}: \overline\alpha^{(0)}_{max} > \overline\alpha^{(1)}_{min} >... >\overline\alpha^{(k_*)}_{max}$. 
According to this algorithm the following inequalities hold:
\[
\overline\alpha^{(0)}_{max} \geq \overline\alpha^{(1)}_{min} >\overline\alpha^{(1)}_{max}>... >\overline\alpha^{(k_*-1)}_{max}> \overline\alpha^{(k_*)}_{min} \geq \overline\alpha^{(k_*)}_{max}.
\]
Note that if $\alpha_M$ is the global minimizer of the function  $\psi_{Q}(\alpha)$ then $\alpha_M \in L^*_{min}$. But in case $\alpha_{MD}< \alpha_{Q}$ the global minimizer of the function  $\psi_{Q}(\alpha)$ may not belong to the set $L^*_{min}$.
For the restricted set of local minimizers the following theorem hold.
\begin{theorem}
The following estimates hold for the local minimum points of the set $L^*_{min}$:
\begin{enumerate}
\item \begin{equation} \label{eq8}
 \min_{\alpha \in L^*_{min}} \n{u_{\alpha}-u_*} \leq \max {\left\{q^{-1}C_1 \min_{\alpha_M \leq \alpha \leq \alpha_0} e_1(\alpha), C_2(b) \min_{\alpha_M \leq \alpha \leq \alpha_0} e_2(\alpha, \delta_*)\right\}} ,   
\end{equation}
where 
\begin{equation}  \label{eq9}
 C_1:=1+\max_{1 \leq k \leq k_*} \max_{\alpha_j  \in \Omega, \overline\alpha^{(k)}_{max} \leq \alpha_j \leq \overline\alpha^{(k-1)}_{max}} T\s{\overline\alpha^{(k)}_{min}, \alpha_j} \leq 1+c_0 c_q\ln \s{\frac{\alpha_0}{\overline \alpha^{(k_*)}_{max}}}
\end{equation}
and $\delta_*=\max{\s{\delta_M,\n{f-f_*}}}$, $C_2(b)=b+2$.\\
\item Let $u_*=\a{A}^pv$, $\n{v}\leq\rho$, $p>0$, $\alpha_0=1$.  If $\delta_0:=\sqrt \alpha_M \leq \n{f-f_*}$, then 
\begin{equation} \label{eq10}
\min_{\alpha \in L^*_{min}} \n{u_{\alpha}-u_*} \leq   c_0 c_p\ln \frac{\n{f-f_*}}{\delta_0} \rho^\frac1{p+1}\a{\ln{\n{f-f_*}}}\n{f-f_*}^\frac p{p+1}, 0<p\leq 2.
\end{equation}
\end{enumerate}
\end{theorem}
\begin{proof}
Due to the inequality $\delta_* \geq \n{f-f_*}$ the global minimizer of the function
$e_2(\alpha,\delta_*)$ is greater or equal to the global minimizer of the function $e_1(\alpha)$. 
Denote $\overline \alpha:=\overline\alpha^{(k_*)}_{min}$, let $\alpha_*$ be the global minimizer of the function $e_2(\alpha,\delta_*)$ and $\alpha_j*$ be the global minimizer of the function  $e_1(\alpha)$ on the set $\Omega$. We consider separately the cases a) $\alpha_j* \geq \overline\alpha$, \quad b) $\alpha_j* \leq \overline\alpha \leq \alpha_*$, \quad c)  $\alpha_* \leq \overline\alpha$. \\
In the case a) we get the estimate 
\begin{equation} \label{eq14}
\min_{\alpha \in L^*_{min}} \n{u_{\alpha}-u_*} \leq q^{-1}C_1 \min_{\alpha_M \leq \alpha \leq \alpha_0} e_1(\alpha)
\end{equation}
analogically to the proof of Theorem 1, but use for the estimation of $T(\alpha^{(k)}_{min}, \alpha_{j})$ the inequality $\Sigma_{j \leq i \leq kmin-1} \frac{\psi_Q(\alpha_{i+1})}{\psi_Q(\alpha_{j})} \leq c_0 M$. \\
In the case b) we estimate 
\begin{equation} \label{eq15}
\n{u_{\overline\alpha}-u_*} \leq \n{u^+_{\alpha_*}-u_*}+0.5{\alpha_j*}^{-1/2}\n{f-f_*} \leq \min_{\alpha \in \Omega} e_1(\alpha) + \min_{\alpha} e_2(\alpha,\delta_*).
\end{equation}
In the case c) we have $\overline\alpha \leq \alpha_{MD}$ and therefore also $\n{B_{\overline\alpha}\s{Au_{\overline\alpha}-f}} \leq b\delta_M  \leq b\delta_*$.
Now we can prove analogically to the proof of the weak quasioptimality of the modified discrepancy principle 
(\cite{RH07}) that under assumption  $\alpha_* \leq \overline\alpha$ the error estimate
\begin{equation}  \label{eq16}
\n{u_{\overline\alpha}-u_*} \leq  C_2(b) \min_{\alpha_M \leq \alpha \leq \alpha_0} e_2(\alpha, \delta_*)
\end{equation}
holds. Now the assertion 1 of Theorem 2 follows from the inequalities
(\ref{eq14}-\ref{eq16}). The proof of assertion 2 is analogical to the proof of Theorem 1.
\end{proof}
We recommend to choose the constant
$b$ from the interval $[1.5;2]$ and coefficient $c_0$ from the interval $[1.5;3]$. In all following numerical examples $b=c_0=2$. The numerical experiments show that the set $L^*_{min}$ contains in many test problems only one local minimizer and this is a good regularization parameter.  In the Table 3 for the test problems \cite{Ha94} the results are given for the set $L^*_{min}$. The columns 2-7 contain the averages and maximums of the error ratio $E$ for the best parameter from the set $L^*_{min}$, the average and maximum of numbers $|L^*_{min}|$ of elements of $L^*_{min}$ and averages and maximums of the constants $C_1$ in the error estimate. The last column of the table contains \% of cases, where the set $L^*_{min}$ contained only one element or two elements one of which was
$\alpha_M$. Tables 2, 3 show that for the best parameter from the set 
$L^*_{min}$ the error ratio $E$ is smaller than for parameter from the ME-rule. Table 3 shows also that in test problems
\emph{foxgood, gravity} ja \emph{ilaplace} the set $L^*_{min}$ contains only one element and this a good parameter. Due to small values of $C_1$ the chosen parameter is pseudooptimal.
Note that average and maximum of the error ratio
$E1$ for the best local minimizer $\alpha_R$ from $L^*_{min}$ over all problems were 0.88 and 1.61 respectively.
\begin{table}
\caption{Results about the set $L^*_{min}$, $p=0$}
\label{tab:1}       
\addtolength\tabcolsep{8.0pt}%
\begin{tabular}{l|cc|cc|cc|c}
\hline\noalign{\smallskip}
Problem & \multicolumn{2}{c|}{Best of $L^*_{min}$}  &\multicolumn{2}{c|}{ $|L^*_{min}|$} &\multicolumn{2}{c|}{ Apost.  $C_1$ } & \%\\
 & Aver E & Max E  & Aver & Max  & Aver & Max\\
\noalign{\smallskip}\svhline\noalign{\smallskip}

Baart & 1.40 & 2.91 & 1.41 & 3 & 6.38 & 7.93 & 60.8 \\
Deriv2 & 1.08 & 1.34 & 2.00 & 2 & 3.54 & 4.49 & 100 \\
Foxgood & 1.57 & 6.69 & 1,00 & 1 & 4.39 & 4.92 & 100 \\
Gravity & 1.14 & 2.15 & 1.00 & 1 & 3.02 & 3.95  & 100 \\
Heat & 1.12 & 2.36 & 2.05 & 3 & 5.08 & 5.38 & 0 \\
Ilaplace & 1.23 & 2.56 & 1.00 &1 & 4.68 & 6.68 & 100 \\
Phillips & 1.06 & 1.72 & 2.10 & 3 & 3.97 & 4.66 & 90.0 \\
Shaw & 1.39 & 3.11 & 1.16 & 2 & 5.89 & 8.06 & 84.2 \\
Spikes & 1.01 & 1.03 & 1.64 & 3 & 10.07 & 11.82 & 55.0 \\
Wing & 1.30 & 1.84 & 2.18 & 4 & 3.03 & 6.63 & 1.7 \\
\noalign{\smallskip}\hline\noalign{\smallskip}
Total  & 1.23 & 6.69 & 1.55 & 4 & 5.01 & 11.82 & 69.2 \\
\noalign{\smallskip}\hline\noalign{\smallskip}
\end{tabular}
\end{table}

\section {Choice of the regularization parameter from the set $L^*_{min}$}

Now we give algorithm for choice of the regularization parameter from the set
$L^*_{min}$.\\
1. If the set $L^*_{min}$ contains only one parameter, we take this for the regularization parameter. On the base of Theorem 2 we know (we can compute also the a posteriori coefficient $C_1$), that this parameter is reliable.\\ 
2. If the set $L^*_{min}$  contains two parameters one of which is  $\alpha_M$, we take for the regularization parameter another parameter $\alpha \not= \alpha_M$. This parameter is good under the assumption that this problem needs regularization.\\
3. If the set $L^*_{min}$ contains after possible elimination of  $\alpha_M$ more than one parameter, we may use for parameter choice the following algorithms. \\
a) Let $\alpha_Q$, $\alpha_{HR}$  be global minimizers of the functions $\psi_{Q}(\alpha)$, $\psi_{HR}(\alpha)$ respectively on the interval $[\max{\s{\alpha_M,\lambda_{min}}},\alpha_0]$. Let  $\alpha_{Q1}:=\max{\s{\alpha_Q, \alpha_{HR}}}$. Choose from the set $L^*_{min}$ the largest parameter $\alpha$, which is smaller or equal to $\alpha_{Q1}$. \\
b)  Let $\alpha_{RE}$ be the global minimizer of the function $\psi_{RE}(\alpha)$ on the interval\\ $[\max{\s{\alpha_M,\lambda_{min}}},\alpha_0]$.  Let  $\alpha_{Q2}$ be the global minimizer of the function $\psi_{Q}(\alpha)$ on the interval $[\alpha_{RE},\alpha_0]$. 
Choose from the set $L^*_{min}$ the largest parameter $\alpha$, which is smaller or equal to $\alpha_{Q2}$. \\
c) For the parameters from
$L^*_{min}$ we compute value $R(\alpha)=\frac{\psi_{HR}(\alpha)}{\n{u_{\alpha}}}$ which we consider as the rough estimate for the relative error 
$\frac{\n{u_{\alpha}-u_*}}{\n{u_*}}$ under assumption that parameter $\alpha$ is near to the optimal parameter.
We choose for the regularization parameter the smallest parameter 
$\alpha_*$ from the set $L^*_{min}$, which satisfies the condition $R(\alpha_*) \leq  C^* \min_{\alpha \in L^*_{min}, \alpha>\alpha_*} R(\alpha)$. We recommend to choose the constant $C^*$ from the interval $5 \leq C^* \leq 10$.  In the numerical experiments we used $C^*=5$. \\
The results of the numerical experiments for different algorithms for the parameter choice are given in the 
Table 4. The results for all 3 algorithms are very similar and the average of the error ratio is even smaller than for $\alpha$ from the ME-rule.
In the case if the set $L^*_{min}$ contained more than 3 parameters, in 68.1 \% of cases all 3 algorithms gave the same parameter and in 92.7 \% of cases the parameters from algorithms b) ja c) coincided. We changed also the parameters $b \in [1.5;2]$ and  $c_0 \in [1.5;3]$, but the overall average of the ratio E changed less than 2 \%. 

\begin{table}
\caption{Averages and maximums of error ratios E in case of different heuristic algorithm, $p=0$}
\label{tab:1}       
%
%
\addtolength\tabcolsep{9.8pt}%
\begin{tabular}{l|cc|cc|cc}
\hline\noalign{\smallskip}
Problem  & \multicolumn{2}{c|}{Algorithm  a)} &\multicolumn{2}{c|}{ Algorithm b)} & \multicolumn{2}{c}{Algorithm c)} \\
 & Aver E & Max E  & Aver E & Max E & Aver E & Max E\\
\noalign{\smallskip}\svhline\noalign{\smallskip}

Baart & 1.83 & 3.63 & 1.61 & 2.91  & 1.61 & 2.91 \\
Deriv2 & 1.08 & 1.34 & 1.08 & 1.34  & 1.08 & 1.34 \\
Foxgood  & 1,57 & 6.69 & 1.57 & 6.69  & 1.57 & 6.69 \\
Gravity & 1.14 & 2.15 & 1.14 & 2.15 & 1.14 & 2.15 \\
Heat & 1.12 & 2.36 & 1.12 & 2.36 & 1.12 & 2.36  \\
Ilaplace & 1.23 & 2.56 & 1.23 & 2.56 & 1.23 & 2.56\\
Phillips & 1.06 & 1.72 & 1.06 & 1.72 & 1.06 & 1.72 \\
Shaw & 1.48 & 3.64 & 1.45 & 3.64 & 1.45 & 3.64\\
Spikes & 1.01 & 1.03 & 1.01 & 1.03 & 1.01 & 1.03 \\
Wing & 1.50 & 1.86 & 1.38 & 2.04 & 1.32 & 1.84\\
\noalign{\smallskip}\hline\noalign{\smallskip}
Total  & 1.30 & 6.69 & 1.26 & 6.69 & 1.26 & 6.69 \\
\noalign{\smallskip}\hline\noalign{\smallskip}
\end{tabular}
\end{table}
The proposed algorithms for parameter choice are complicated (formation of the set 
$L^*_{min}$) but they enable to estimate also the reliability of the chosen parameter and propose alternative parameters if the set 
 $L^*_{min}$ contains several local minimizers. If some information about solution or noise is available, it may help to find from the set
$L^*_{min}$  better parameter than algorithms a)-c) find. If the purpose is only parameter choice, simpler rules below may be used (parameters $\alpha_{Q1}$ and $\alpha_{Q2}$ are defined in algorithm a), b)). \\ 
1. 
We choose for the regularization parameter the smallest local minimizer $\alpha^{(k_*)}_{min}$ of the function
$\psi_{Q}(\alpha)$ which satisfies the following conditions: 
\be{c1}
\frac{\psi_{Q}(\alpha^{(k)}_{max})}{\psi_{Q}(\alpha^{(k)}_{min})} \leq c_0, \qquad k=k_0,k_0+1,...,k_*-1; 
\ee
\be{c2}
\frac{\psi_{Q}(\alpha^{(k)}_{min})}{\min_{j \leq k} {\psi_{Q}(\alpha^{(j)}_{min})}} \leq c_0 , \qquad k=k_0,k_0+1,...,k_* ,
\ee
where $k_0$ is the index for which
$\alpha^{(k_0)}_{min} \leq \alpha_{Q1} \leq  \alpha^{(k_0-1)}_{max} $.\\
2.  
We choose for the regularization parameter the smallest local minimizer $\alpha^{(k_*)}_{min}$ of the function $\psi_{Q}(\alpha)$ satisfying conditions (\ref{c1}), (\ref{c2}) where  $k_0$ is index for which $\alpha^{(k_0)}_{min} \leq \alpha_{Q2} \leq  \alpha^{(k_0-1)}_{max} $. \\
These rules give in test problems \cite{Ha94} the same results as the algorithms a) and b) respectively. 

The Table 5 gives results of the numerical experiments in the case of smooth solution,
 $p=2$. The table shows that in case of smooth solution the number of local minimizers in $L_{min}$ and number of elements $L^*_{min}$ are smaller than in case
$p=0$. If the set $L^*_{min}$ contains several elements, 
then the algorithms a) and c) gave the same parameter, which was always the best parameter from $L^*_{min}$ with smallest error. In case of algorithm b) the overall average of the ratio $E$ was 1.25. In all problems except 
the problem \emph{wing} the heuristical rule gave parameter where the average of error was smaller than by parameter from the ME-rule, and only 10 \% larger than by parameter from the MEe-rule (both ME-rule and the MEe rule used the exact noise level). 
\begin{table}
\caption{Results of the numerical experiments, $p=2$}
\label{tab:1}       
%
%
\addtolength\tabcolsep{2.8pt}%
\begin{tabular}{lc|c|c|c|c|c|c}
\hline\noalign{\smallskip}
Problem & ME & MEe  & Best of $L_{min}$ & $|L_{min}|$&Best of $L^*_{min}$ & $|L^*_{min}|$ & \%\\
 & Aver E & Aver E & Aver E  & Aver & Aver E & Aver \\
\noalign{\smallskip}\svhline\noalign{\smallskip}

Baart & 1.86  & 1.19 & 1.18 & 4.74 & 1.41 & 1.02 & 98.3\\
Deriv2 & 1.10 &  1.19 & 1.03 & 2.00 & 1.03 & 2.00 & 100\\
Foxgood & 1.56 &  1.13 & 1.14 & 2.08 & 1.20 & 1.00  & 100 \\
Gravity & 1.33  &  1.05 & 1.09 & 1.72 & 1.11 & 1.00  & 100\\
Heat & 1.13 & 1.12 & 1.05 & 2.10 & 1.05 & 2.10  & 0\\
Ilaplace & 1.47 & 1.06 & 1.11 & 2.73 & 1.11 & 1.00  & 100\\
Phillips & 1.26 & 1.06 & 1.04 & 2.10 & 1.04 & 2.10  & 90\\
Shaw & 1.37 & 1.06 & 1.11 & 3.72 & 1.22 & 1.01  & 99.2\\
Spikes & 1.85  & 1.12 & 1.19 & 4.78 & 1.31 & 1.00  & 100 \\
Wing & 1.67 & 1.14 & 1.22 & 4.53 & 1.73 & 1.01  & 99.2 \\
\noalign{\smallskip}\hline\noalign{\smallskip}
Total  & 1.46 & 1.11 & 1.12 & 3.05 & 1.22 & 1.32  & 88.7\\
\noalign{\smallskip}\hline\noalign{\smallskip}
\end{tabular}
\end{table}

We finish the paper with the following conclusion. For the heuristic choice of the regularization parameter we recommend to choose the parameter from the set of local minimizers of the function $\psi_Q(\alpha)$.
Proposed algorithm enables to restrict this set and in many problems the restricted set contains only one element, this parameter is the pseudooptimal parameter.

\begin{acknowledgement}
The authors are supported 
by institutional research funding IUT20-57
of the Estonian Ministry of Education and Research. 
 \end{acknowledgement}

\end{document}